\documentclass[reqno]{amsart}[12pt]
\usepackage{amsmath}
\usepackage{enumerate}
\usepackage{color}
\newtheorem{theorem}{Theorem}[section]
\newtheorem{lemma}[theorem]{Lemma}

\begin{document}
\title[\hfil Uniqueness of quasilinear equation]
{Uniqueness of positive solution for a quasilinear elliptic equation in Heisenberg Group}

\author[K. Bal]
{Kaushik Bal}  

\address{Kaushik Bal \newline
Department of Mathematics and Statistics\\
Indian Institute of Technology\\
Kanpur\\
Uttar Pradesh-280016 India}
\email{kaushik@iitk.ac.in}

%\thanks{Submitted July 24, 2013. Published November 8, 2013.}

\begin{abstract}
 In this article we are interested in addressing the question of existence and uniqueness of positive solution to a quasilinear elliptic equation involving
 p-laplacian in Heisenberg Group. The idea is to prove the uniqueness by using Diaz-Sa\'a Inequality in Heisenberg 
 Group which we obtain via a generalized version of Picone's Identity.
 
\end{abstract}

\maketitle
\numberwithin{equation}{section}
\newtheorem{remark}[theorem]{Remark}
\allowdisplaybreaks

\section{Introduction}
In this article we will be dealing with the problem of existence and uniqueness of the p-sub-laplacian operator in Heisenberg Group.
We begin by recalling the well-known paper of H.Br\'ezis and L.Oswald \cite{BrOs}, where the necessary and sufficient condition for 
the existence and uniqueness of positive solutions were obtained for the equation:
\begin{equation}
 \nonumber
 -\Delta u=f(x,u)\;\text{in}\;\Omega, \;\;u=0\;\text{on}\;\partial\Omega
\end{equation}
when $\Omega$ is a bounded open domain in $\mathbb{R}^n$.\\
Almost immediately the result was extended in J.Di\'az and J.Sa\'a \cite{DiSa} to the p-laplacian case by introducing a new 
inequality which came to be later known as the D\'iaz-Sa\'a Inequality.

The purpose of this paper is to extend the result of \cite{DiSa} in the context of Heisenberg Group, which we will denote by 
$\mathcal{H}^n$.\\
Consider the problem:
\begin{align}\label{QH}
 -\Delta_p u&=f(x,u)\;\text{in}\;\Omega\nonumber\\
u\geq 0\;&\text{and}\;u\not\equiv 0\;\text{in}\;\Omega\\
 u&=0\;\text{on}\;\partial\Omega\nonumber
\end{align}
where $\Omega$ is an open bounded domain of $\mathcal{H}^n$ and $1<p<\infty$.\\
We consider $f:\Omega\times[0,\infty)\to (0,\infty)$ satisfying the following hypothesis: 
\begin{enumerate}[I]
 \item The function $r\rightarrow f(x,r)$ is continuous on $[0,\infty)$ for a.e $x\in\Omega$ and for every $r\geq0$, the function
 $x\rightarrow f(x,r)$ is in $L^{\infty}(\Omega)$.
 \item The function $r\rightarrow \frac{f(x,r)}{r^{p-1}}$ is strictly decreasing on $(0,\infty)$ for a.e $x\in\Omega$.
 \item $\exists\;\;C>0$ s.t $f(x,r)\leq C(r^{p-1}+1)$ for a.e $x\in\Omega$ and for all $r$.
\end{enumerate}

We aim to establish the existence and uniqueness of the weak solution to (\ref{QH}).

Before we start with our results let us briefly recall some basic notions regarding the Heisenberg Group $\mathcal{H}^n$ along 
with some literature which is available on the study of Elliptic Equation on Heisenberg Group.

The Heisenberg Group $\mathcal{H}^n=(\mathcal{R}^{2n+1},\cdot)$ is Nilpotent Lie Group endowed with the group structure:
\begin{equation}\nonumber
 (x,y,t)\cdot(x',y',t')=(x+x',y+y',t+t'+2(\langle y,x'\rangle-\langle x,y'\rangle))
\end{equation}
where $x,y,x',y'\in \mathcal{R}^n,\;t,t'\in\mathcal{R}$ and $\langle,\rangle$ denotes the standard inner product in $\mathcal{R}^n$.

The left invariant vector field generating the Lie algebra is given by
\begin{equation}
 \nonumber
 \mathcal{T}=\frac{\partial}{\partial t},\,\mathcal{X}_i=\frac{\partial}{\partial x_i}+2y_i\frac{\partial}{\partial t},\,
 \mathcal{Y}_i=\frac{\partial}{\partial y_i}-2x_i\frac{\partial}{\partial t},\,i=1,2,\cdot,\cdot,n.
\end{equation}
and satisfy the following relationship:
\begin{equation}\nonumber
 [\mathcal{X}_i,\mathcal{Y}_i]=-4\delta_{ij}T,\;\;[\mathcal{X}_i,\mathcal{X}_j]=[\mathcal{Y}_i,\mathcal{Y}_j]=
 [\mathcal{X}_i,\mathcal{T}]=[\mathcal{Y}_i,\mathcal{T}]=0.
\end{equation}

The generalized gradient is given by $\nabla_H=(\mathcal{X}_1,\mathcal{X}_2,\cdot,\cdot,\mathcal{X}_n,\mathcal{Y}_1,\mathcal{Y}_2,
\cdot,\cdot,\mathcal{Y}_n)$.\\
Hence the sub-laplacian $\Delta_H$ and the p-sub-Laplacian $\Delta_{H,p}$ are denoted by

\begin{align*}
 \Delta_H=\sum\limits_{i=1}^{n}\mathcal{X}_i^2+\mathcal{Y}_i^2=\nabla_H\cdot\nabla_H,\;\mbox{and}\\
 \Delta_{H,p}=\nabla_H(|\nabla_H |^{p-2}\nabla_H ),\;p>1
\end{align*}
We also denote the space $D^{1,p}(\Omega)$ and $D^{1,p}_0(\Omega)$ as
$\{u:\Omega\rightarrow \mathcal{R};u,|\nabla_H u|\in L^p(\Omega)\}$ and the closure of $C^{\infty}_0(\Omega)$ with respect to the norm
$||u||_{D^{1,p}_0(\Omega)}=(\int_{\Omega}|\nabla_H u|^p dxdydt)^{\frac{1}{p}}$
respectively.

For more details about Heisenberg Group the reader may consult \cite{CDPT}.

Some results on the Laplacian and the p-Laplacian has been generalized to the Heisenberg Group with various degree of success.
Consider the following problem:
\begin{equation}\label{NIU}
%\nonumber
\begin{array}{l}
-\Delta_{H,p} u=f(u)\;\text{in}\;\Omega\\
%u\geq 0\;\text{and}\;u\not\equiv 0\;\text{in}\;\Omega\\
 u=0\;\text{on}\;\partial\Omega
 \end{array}
\end{equation}
Some of the very first results obtained regarding the above problem for $p=2$ is by Garofalo-Lanconelli \cite{GL}, where
existence and nonexistence results were derived using integral identities of Rellich-Pohozaev type. In Birindelli et al \cite{BDC},
Liouville theorems for semilinear equations are proved. One can also find monotonicity and symmetry results in Birindelli and Prajapat \cite{BP}. 
As for the p-sub-Laplacian case, Niu et al \cite{NZW} considered the question of non-uniqueness of the (\ref{NIU})
using the Picone Identity and the Pohozaev Identities for the p-sub-Laplacian on Heisenberg Group. Results on p-sub-Laplacian
involving singular indefinite weight can be found in J.Dou \cite{DOU} and J.Tyagi \cite{TY} and the reference therein.

One of the biggest problem when dealing with p-sub-Laplacian is the non-availability of the $C^{1,\alpha}$ regularity, although
it has been proved in Marchi \cite{Marchi} to exist for $p$ near $2$.
It is worth mentioning that the methods of D\'iaz-S\'aa \cite{DiSa} can't be directly applied here due to the non-availability
of $C^{1,\alpha}$ regularity in the Heisenberg Group. In this work we bypass that problem by using a generalized 
version of D\'iaz-S\'aa Inequality in Heisenberg Group. 
\section{Preliminary Results}
We start this section with the generalized Picone's Identity for $p$-sub-Laplacian in Heisenberg Group, which is extension of the main
result in Euclidean space obtained in \cite{KB}.\\
In what follows we will assume $g:(0,\infty)\rightarrow(0,\infty)$ is a locally Lipchitz function that satisfies the differential inequality:
 \begin{equation}\label{PC}
  g'(x)\geq (p-1)[g(x)]^{\frac{p-2}{p-1}}\;\text{a.e}\;\;\text{in}\;\;(0,\infty)
 \end{equation}
 
 \begin{remark}
Example of functions satisfying (\ref{PC}) is $g(x)=x^{p-1}$ (where the equality holds) and $e^{(p-1)x}$.  
\end{remark}

In what follows we will use $\nabla$ to denote $\nabla_H$ and $\Delta_p$ to denote $\Delta_{H,p}$.

\begin{theorem}(\textbf{Generalized Picone Identity})\label{PICN}
 Let $1<p<\infty$ and $\Omega$ be any domain in $\mathcal{H}^n$. Let $u$ and $v$ be differentiable functions on $\Omega$
 with $v>0$ a.e in $\Omega$. Also assume $g$ satisfies (\ref{PC}).
 %Also assume that $f'(y)\geq (p-1)[f(y)^{\frac{p-2}{p-1}}]$ for all $y$. 
 Define
 \begin{gather*}
 L(u,v)=|\nabla u|^p-p\frac{|u|^{p-2}u}{g(v)}\nabla u\cdot\nabla v|\nabla v|^{p-2}
 +\frac{g'(v)|u|^p}{[g(v)]^2}|\nabla v|^p\;\;\text{a.e}\;\text{in}\;\Omega.\\
 R(u,v)=|\nabla u|^p-\nabla(\frac{|u|^p}{g(v)})|\nabla v|^{p-2}\nabla v\;\;\text{a.e}\;\text{in}\;\Omega.
 \end{gather*}
Then $L(u,v)=R(u,v)\geq0$. Moreover $L(u,v)=0$ a.e. in $\Omega$ 
if and only if $\nabla (\frac{u}{v})=0$ a.e. in $\Omega$.
\end{theorem}

\begin{remark}
Note that there is no restriction on the sign of $u$, as one can find in Proposition 3 of \cite{Ch}. When $g(x)=x^{p-1}$ and $u\geq 0$, 
we get back Lemma 2.1 (i.e, Picone Identity) of \cite{DOU}.
\end{remark}

\begin{proof}[Proof of Theorem \ref{PICN}]
Expanding $\nabla(\frac{|u|^p}{g(v)})$ we have,
\begin{align*}
\nabla(\frac{|u|^p}{g(v)})=\frac{pg(v)|u|^{p-2}u\nabla u-g'(v)|u|^p\nabla v}{[g(v)]^2}\\
=p\frac{|u|^{p-2}u\nabla u}{g(v)}-\frac{g'(v)|u|^p\nabla v}{[g(v)]^2}.
\end{align*}
Plugging it in $R(u,v)$ we have $R(u,v)=L(u,v)$.\\
To show positivity of $L(u,v)$ we proceed as follows,
\begin{equation}\label{mod}
 %\nonumber
 \frac{|u|^{p-2}u}{g(v)}\nabla u.\nabla v|\nabla v|^{p-2}\leq \frac{|u|^{p-1}}{g(v)}|\nabla v|^{p-1}|\nabla u|
\end{equation}
and by Young's Inequality we have,
\begin{equation}\label{young}
 p\frac{|u|^{p-1}}{g(v)}|\nabla v|^{p-1}|\nabla u|\leq |\nabla u|^p+(p-1)\frac{|u|^p|\nabla v|^p}{[g(v)]^{\frac{p}{p-1}}}
\end{equation}

Using (\ref{mod}) and (\ref{young}) we have,
\begin{align*}
 L(u,v)&\geq -(p-1)\frac{|u|^p|\nabla v|^p}{[g(v)]^{\frac{p}{p-1}}}+\frac{g'(v)|u|^p}{[g(v)]^2}|\nabla v|^p
\end{align*}
Now since $g$ satisfies (\ref{PC}) i.e, $g'(x)\geq (p-1)[g(x)]^{\frac{p-2}{p-1}}$ we have, $L(u,v)\geq0$.

Equality holds when the following occurs simultaneously:
\begin{align}
g'(x)&= (p-1)[g(x)]^{\frac{p-2}{p-1}}\label{PEC}\\
\frac{|u|^{p-2}u}{g(v)}\nabla u.\nabla v|\nabla v|^{p-2}&= \frac{|u|^{p-1}}{g(v)}|\nabla v|^{p-1}|\nabla u|.
\end{align}
and,
\begin{equation}\label{hi}
|\nabla u|=\frac{|u\nabla v|}{{g(v)}^{\frac{1}{p-1}}}
\end{equation}

Set, 
\begin{equation}\nonumber
 \mathcal{X}=\{x\in\Omega:\frac{|u\nabla v|}{{g(v)}^{\frac{1}{p-1}}}=0\}
\end{equation}

By equation (\ref{hi}) we have,

\begin{equation}\label{ko}
 \frac{|u\nabla v|}{{g(v)}^{\frac{1}{p-1}}}=|\nabla u|=0\;\;\text{a.e}\;\;\text{on}\;\;\mathcal{X}.
\end{equation}

from (\ref{ko}) and (\ref{PEC}) we have for $g(x)=x^{p-1}$,
\begin{equation}\label{pii}
 \frac{u}{v}\nabla v=\nabla u=0\;\; \text{a.e}\;\; \text{on}\;\; \mathcal{X}.
\end{equation}

On $\mathcal{X}^c$, let
\begin{equation}\nonumber
 w=\frac{|\nabla u|[g(v)]^{\frac{1}{p-1}}}{|u\nabla v|}
\end{equation}

Hence from the fact that $L(u,v)=0$ a.e in $\Omega$ we have,
\begin{equation}\label{odd}\nonumber
 w^p-pw+p-1=0
\end{equation}

which holds iff $w=1$.\\
Again taking into account (\ref{PEC}) is true for $g(x)=x^{p-1}$ we have,
\begin{equation}\label{pi}
 \nabla u\cdot\big(\nabla u-\nabla v\frac{u}{v})=0\;\;\text{a.e}\;\;\text{in}\;\;\mathcal{X}^c
\end{equation}
Combining (\ref{pii}) and (\ref{pi}) we can easily conclude that $L(u,v)=0$ iff $\nabla\big(\frac{u}{v})=0$ a.e in $\Omega$.
\end{proof}

With the generalized Picone's Identity in our hands we can now proceed to prove the Picone's Inequality which is the vital
ingredient for the proof of D\'iaz-Sa\'a Inequality. We will present a non-linear version of the Inequality and will closely follow 
the proof of Abdellaoui-Peral \cite{AbPe}.

\begin{theorem}(\textbf{Generalized Picone Inequality})\label{PIn}
Let $p>1$ and $\Omega$ be a bounded domain in $\mathcal{H}^n$. If $u,v\in D^{1,p}_0(\Omega)$ s.t $-\Delta_p v=\mu$ where $\mu$ is a positive bounded Radon 
measure with $v|_{\partial\Omega}=0$, $v(\not\equiv 0)\geq0$ and $g$ satisfies (\ref{PC}). Then we have,
\begin{equation}
 \nonumber
 \int_{\Omega} |\nabla u|^p\geq \int_{\Omega}\Big(\frac{|u|^p}{g(v)}\Big)(-\Delta_p v).
\end{equation}
\end{theorem}
\begin{remark}
 When $g(u)=u^{p-1}$, we get Picone's Inequality in Heisenberg Group in Dou \cite{DOU}. 
\end{remark}

Before we proceed with the proof of our theorem we need the following lemma:
\begin{lemma}\label{p1}
 Let $p>1$ and $\Omega$ be any domain in $\mathcal{H}^n$ and let $v\in D^{1,p}(\Omega)$ be such that $v\geq\delta>0$. Then for all 
 $u\in C_c^{\infty}(\Omega)$ we have,
\begin{equation}\nonumber
 \int_{\Omega}|\nabla u|^p\geq \int_{\Omega}\Big(\frac{|u|^p}{g(v)}\Big)(-\Delta_p v).
\end{equation}
\end{lemma}
\begin{proof}
 Since $v\in D^{1,p}(\Omega)$, we can by Meyers-Serrin Theorem, choose $v_n\in C^1(\Omega)$ such that the following holds:
 \begin{equation}\nonumber
  v_n>\frac{\delta}{2}\;\;\text{in}\;\;\Omega,\;\;
  v_n\rightarrow v\;\;\text{in}\;\;\Omega\;\;\text{and}\;\;
  v_n\rightarrow v\;\;\text{a.e}\;\;\text{in}\;\;\Omega.
 \end{equation}
 Employing Theorem \ref{PICN} with $v_n$ and $u$ we have,
 \begin{equation}\nonumber
  \int_{\Omega}R(u,v_n)=0\;\;\text{since}\;\;R(u,v_n)=0\;\;\text{a.e}\;\text{in}\;\Omega\;\;\text{and}\;\forall\;n\in\mathbf{N}.
 \end{equation}
 i.e,
 \begin{eqnarray}\nonumber
   \int_{\Omega}|\nabla u|^p\geq \int_{\Omega}\nabla\Big(\frac{|u|^p}{g(v_n)}\Big)|\nabla v_n|^{p-2}\nabla v_n
 =\int_{\Omega}\frac{|u|^p}{g(v_n)}(-\Delta_p v_n).\nonumber
 \end{eqnarray}

Note that since $-\Delta_p$ is a continuous function from $D^{1,p}(\Omega)$ to $D^{-1,p'}(\Omega)$ for $p'=\frac{p}{p-1}$,
we have $-\Delta_p v_n\rightarrow -\Delta_p v$ in $D^{1,p}(\Omega)$ and for $g$ locally lipchitz continuous in $(0,\infty)$
we have $g(v_n)\rightarrow g(v)$ pointwise.\\
Hence using Lebesgue Dominated Convergence Theorem and the fact that $g$ is increasing on $(0,\infty)$, 
we have,
\begin{equation}\nonumber
 \int_{\Omega}|\nabla u|^p\geq \int_{\Omega}\frac{|u|^p}{g(v)}(-\Delta_p v)
\end{equation}
for any $u\in C^{\infty}_c(\Omega)$. 
\end{proof}

Before we proceed with the proof of Theorem \ref{PIn}, we will state the Strong Maximum Principle from \cite{DOU} which was proved using the Harnack 
Inequality of \cite{CaDaGa}.

\begin{lemma}\label{smp}(\textbf{Strong Maximum Principle})
 Let $p>1$ and $\Omega\subset \mathcal{H}^n$ be a bounded domain and $u\in D^{1,p}_0(\Omega)$ be nonnegative solution of the following
 equation
 \begin{equation}\nonumber
  -\Delta_p u=h(x,u)\;\;\text{in}\;\;\Omega;\;u|_{\partial\Omega}=0
 \end{equation}
 where $h:\Omega\times\mathbb{R}\rightarrow\mathbb{R}$ is a measurable function such that $|h(x,u)|\leq C(u^{p-1}+1)$. 
 Then $u\equiv0$ or $u>0$ in $\Omega$.
\end{lemma}

With Lemma \ref{p1} and Lemma \ref{smp} in hand we now proceed with the proof of the Theorem \ref {PIn}.

\begin{proof}
 Using the Strong Maximum Principle we have $v>0$ in $\Omega$. Denote, $v_n(x)=v(x)+\frac{1}{n},\;n\in\mathbb{N}$. 
 Thus we have the following:
 \begin{itemize}
  \item $\Delta_p v_m=\Delta_p v$.
  \item $v_n\rightarrow v$ a.e in $\Omega$ and in $D^{1,p}(\Omega)$.
  \item $g(v_n)\rightarrow g(v)$ a.e in $\Omega$.
 \end{itemize}
Hence using Lemma \ref{p1} we have for $u\in C^{\infty}_c(\Omega)$,
\begin{equation}\nonumber
\int_{\Omega}|\nabla u|^p\geq \int_{\Omega}\frac{|u|^p}{g(v)}(-\Delta_p v)
\end{equation}
Now to conclude our theorem for $u\in D_0^{1,p}(\Omega)$, we use $u_n\in C^{\infty}_0(\Omega)$ such that $u_n\rightarrow u$ in 
$D^{1,p}_0(\Omega)$. Choosing $u_n$ and $v_n$ in Lemma \ref{p1}, we have
\begin{equation}\nonumber
 \int_{\Omega}|\nabla u_n|^p\geq \int_{\Omega}\Big(\frac{|u_n|^p}{g(v_n)}\Big)(-\Delta_p v_n)
\end{equation}
Now using the fact that $g$ satisfies (\ref{PC}) we have by Fatou Lemma,
\begin{equation}\nonumber
 \int_{\Omega}|\nabla u|^p\geq \int_{\Omega}\Big(\frac{|u|^p}{g(v)}\Big)(-\Delta_p v)
\end{equation}
which concludes our proof.
\end{proof}
We will conclude this section with the D\'iaz-Sa\'a Inequality in Heisenberg Group. For this part we will be using $g(u)=u^{p-1}$.
\begin{theorem}\label{diazsaa}(\textbf{D\'iaz-Sa\'a Inequality})
Let $p>1$ and $\Omega$ be a bounded domain in $\mathcal{H}^n$. If $u_i\in D^{1,p}_0(\Omega)$ s.t $-\Delta_p u_i=\mu_i$, where $\mu_{i}>0$ is a bounded Radon 
measure with $u_i|_{\partial\Omega}=0$ and $u_i(\not\equiv 0)\geq0\;\text{a.e in}\;\Omega$ for $i=1,2$. Then we have,
\begin{equation}\nonumber
\int_{\Omega} \Big(-\frac{\Delta_p u_1}{u_1^{p-1}}+\frac{\Delta_p u_2}{u_2^{p-1}}\Big)(u_1^p-u_2^p)\geq0.
\end{equation}
\end{theorem}
\begin{remark}
Note that above theorem is not true for a general $g$ satisfying (\ref{PC}).
\end{remark}

\begin{proof}
Choosing $u_i$ for $i=1,2$ satisfying the hypothesis of Theorem \ref{diazsaa} and then plugging the tuple $(u_1,u_2)$ into Theorem \ref{PIn} we get
%  \int_{\Omega} \Big(-\frac{\Delta_p u}{g(u)}+\frac{\Delta_p v}{g(v)}\Big)u^p\geq 0.
\begin{equation}
 \nonumber \int_{\Omega} |\nabla u_1|^p\geq \int_{\Omega}\Big(-\frac{\Delta_p u_2}{u_2^{p-1}}\Big)u_1^p
 \end{equation}
Using Integration by Parts on right part, we have
\begin{equation}\label{df}
 \int_{\Omega}\Big(-\frac{\Delta_p u_1}{u_1^{p-1}}+\frac{\Delta_p u_2}{u_2^{p-1}}\Big)u_1^p\geq0.
\end{equation}

Now interchanging the tuple $(u_1,u_2)$ into $(u_2,u_1)$ in Theorem \ref{PIn} we get,
\begin{equation}
 \nonumber \int_{\Omega} |\nabla u_2|^p\geq \int_{\Omega}\Big(-\frac{\Delta_p u_1}{u_1^{p-1}}\Big)u_2^p
 \end{equation}
Again using Integration by Parts on the right part we have,
\begin{equation}\label{dfg}
 \int_{\Omega}\Big(-\frac{\Delta_p u_2}{u_2^{p-1}}+\frac{\Delta_p u_1}{u_1^{p-1}}\Big)u_2^p\geq0.
\end{equation}
Adding (\ref{df}) and (\ref{dfg}) we have,
\begin{equation}\nonumber
 \int_{\Omega}\Big(-\frac{\Delta_p u_1}{u_1^{p-1}}+\frac{\Delta_p u_2}{u_2^{p-1}}\Big)(u_1^p-u_2^p)\geq0.
\end{equation}
Hence proved.
\end{proof}

\section{Main Results}
In this section of the paper we will state and proof our main result. 
\begin{theorem}\label{uni}(\textbf{Uniqueness of Solution})
There exists at most one positive weak solution to equation (\ref{PC}) in $D^{1,p}_0(\Omega)$.
\end{theorem}

\begin{proof}
 Let $u$ and $v$ be two non-positive solutions of equation (\ref{QH}). Then using Lemma \ref{smp} we have, $u,v>0$ in $\Omega$.
 Moreover since $f(x,u)$ is positive and satisfy hypothesis (I), we have for $u\neq v$,
\begin{equation}\nonumber
0\geq\int_{\Omega}\Big(-\frac{\Delta_p u}{u^{p-1}}+\frac{\Delta_p v}{v^{p-1}}\Big)(u^p-v^p)=\int_{\Omega}\Big(\frac{f(x,u)}{u^{p-1}}-
\frac{f(x,v)}{v^{p-1}}\Big)(u^p-v^p)<0.
\end{equation}
Hence we arrive at a contradiction.
\end{proof}

We now proceed to the proof of existence of solution in $D^{1,p}_0(\Omega)$. The energy functional corresponding to equation (\ref{QH})
as 
\begin{equation}\label{ef}
\nonumber
E(u)=\frac{1}{p}\int_{\Omega}|\nabla u|^p-\int_{\Omega}F(x,u)\;\mbox{with}\;u\in D^{1,p}_0(\Omega)
\end{equation}

where $F(x,u)=\int_{0}^{u}f(x,t)dt$.

\begin{theorem}\label{exi}(\textbf{Existence of Solution})
 The equation (\ref{PC}) admits a solution if the following two conditions hold simultaneously:
 \begin{eqnarray}
 \lambda_1(-\Delta_p v-a_o |v|^{p-2}v)<0\;\text{with}\;a_o(x)=\lim_{r\searrow 0} \frac{f(x,r)}{r^{p-1}}\\
 \lambda_1(-\Delta_p v-a_{\infty} |v|^{p-2}v)>0\;\text{with}\;a_{\infty}(x)=\lim_{r\nearrow \infty} \frac{f(x,r)}{r^{p-1}}.
 \end{eqnarray}
\end{theorem}

Using the exact same proof of \cite{DiSa} (Th\'eor\'eme 2), one can show the existence of the minimizer to $E(u)$ in $D^{1,p}_0(\Omega)$
and hence Theorem \ref{exi}.   

\section*{Comments}
We conclude this article with a few comments:
\begin{itemize}
 \item It is worth mentioning again that the methods used in Di\'az-Sa\'a \cite{DiSa} for proving the existence and uniqueness of equation (\ref{QH}) can't be used here due to non-availability
of the $C^{1,\alpha}$ regularity of the $p$-sub-Laplacian in Heisenberg Group for all $p>1$, so we have used the Picone Inequality to bypass the problem but in doing so
we are forced to put the positivity condition on $f$, which was not present in the assumptions of \cite{DiSa}.\\
It will be interesting to know if one can conclude the same results for uniqueness without the positivity condition on $f$.
\item The statements proved in Theorem \ref{PICN} and Theorem \ref{PIn} were valid for a wide range of functions satisfies (\ref{PC}). This results are
 new in the contexts of Heisenberg group and one can actually obtain Nonexistence results and Comparison Principles for p-sub-Laplacian
similar to those in \cite{KB} and the reference therein.
\end{itemize}

\subsection*{Acknowledgements} 
The work has been carried under the project INSPIRE Faculty Award [IFA-13 MA-29].

\end{document}